\documentclass{article}
\usepackage{relsize,calc,bigstrut,afterpage,rotfloat,multirow,array,amsmath,amsthm,amssymb,mathrsfs}
\usepackage{hyperref}
\usepackage[plain]{algorithm}
\usepackage{algpseudocode}
\usepackage[all]{xy}

\newcommand{\Star}{\ast}

\newtheorem{theorem}{Theorem}[section]
\newtheorem{definition}[theorem]{Definition}

\newtheorem{fact}[theorem]{Fact}

\newtheorem*{question}{Question}

\newcommand{\A}{\mathscr{A}}

\newcommand{\Q}{\mathcal{Q}}
\renewcommand{\P}{\mathcal{P}}

\newcommand{\SP}{\mathscr{P}}
\newcommand{\QP}{(\Q,\P)}

\newcommand{\modx}[1]{\ (\mathrm{mod}~#1)}

\newcommand{\<}{\langle}
\renewcommand{\>}{\rangle}

\newcommand{\llex}{<_\mathrm{lex}}
\newcommand{\PPos}{$\SP$-position}

\newcommand{\CGSuite}{{\itshape\textsmaller{C\kern- 0.06emG\kern-0.1emS\kern-0.04em}uite}}
\newcommand{\sh}{\textrm{\protect\raisebox{0.85pt}{\tiny /}}}

\input macros.tex

\begin{document}

\title{Mis\`ere Quotients for Impartial Games\\ {\small SUPPLEMENTARY MATERIAL}}

\author{Thane E. Plambeck\\2341 Tasso St.\\Palo Alto, CA 94301 \and Aaron N. Siegel\\Institute for Advanced Study\\1 Einstein Drive\\Princeton, NJ 08540}



\date{\today}
\maketitle

These are the supplementary appendices to the paper, Mis\`ere quotients for impartial games~\cite{siegel_200Xd}.  If you have not read the actual paper, you should do so \emph{before} reading this supplement, or else it will not make much sense!

Appendix~\ref{appendix:details} gives detailed solutions to many of the octal games discussed in the paper, and Appendix~\ref{appendix:algorithms} describes the algorithms used to compute most of our solutions.

\appendix
\setcounter{section}{1}


\section{Solutions in Detail}

\label{appendix:details}

This appendix contains detailed solutions to many of the games discussed in Appendix~A.


Figure~\ref{figure:3digitsummary} summarizes the status of every octal game with at most three code digits.  For each game $\Gamma$, the chart indicates whether $\Gamma$ is tame or wild, and whether its normal- and/or mis\`ere-play solution is known.

Figures~\ref{figure:simpleoctals} and~\ref{figure:053} present complete solutions to wild two- and three-digit octal games with relatively simple mis\`ere quotients and pretending functions.  Figure~\ref{figure:quaternaries} does the same for wild four-digit quaternaries.  Finally, Figures~\ref{figure:0115},~\ref{figure:0152} and~\ref{figure:077} present the solutions to \textbf{0.115}, \textbf{0.152} and \textbf{0.77}, respectively.  Note that $\Q(\textbf{0.15}) \cong \Q(\textbf{0.115})$.  The solution to \textbf{0.644} is omitted due to its size; see~\cite{miseregames_www}.

\newcommand{\TK}{$\oplus$}
\newcommand{\WK}{$\ominus$}
\newcommand{\WU}{$-$}
\newcommand{\WQ}{$\star$}

\begin{figure}[p]
\centering
{\small
\begin{tabular}{cc}
\begin{tabular}{>{\bfseries}c|c@{\hspace{0.15cm}}c@{\hspace{0.15cm}}c@{\hspace{0.15cm}}c@{\hspace{0.15cm}}c@{\hspace{0.15cm}}c@{\hspace{0.15cm}}c@{\hspace{0.15cm}}c}
 & \multicolumn{8}{c}{$\mathbf{d_3}$} \\
$\mathbf{d_1d_2}$ & \textbf{0} & \textbf{1} & \textbf{2} & \textbf{3} & \textbf{4} & \textbf{5} & \textbf{6} & \textbf{7} \\ \hline
.00 & \TK & \TK & \TK & \TK & \WU & \WU & \WU & \WU \\
.01 & \TK & \TK & \TK & \TK & \WU & \WU & \WU & \WQ \\
.02 & \TK & \TK & \TK & \TK & \WU & \WU & \WU & \WU \\
.03 & \TK & \TK & \TK & \TK & \WU & \WU & \WU & \WU \\
.04 & \WU & \WU & \WU & \WQ & \WK & \WQ & \WK & \WQ \\
.05 & \TK & \WQ & \TK & \WQ & \WQ & \WQ & \WU & \WQ \\
.06 & \WU & \WU & \WU & \WU & \WU & \WU & \WU & \WU \\
.07 & \WQ & \WQ & \WQ & \WQ & \WK & \WK & \WK & \WK \\
.10 & \TK & \TK & \TK & \TK & \WQ & \TK & \WQ & \TK \\
.11 & \TK & \TK & \TK & \TK & \WU & \WK & \WQ & \WQ \\
.12 & \TK & \TK & \TK & \WK & \WQ & \WU & \WU & \WQ \\
.13 & \TK & \TK & \TK & \TK & \WQ & \WU & \WU & \WQ \\
.14 & \WU & \WQ & \WU & \WU & \WK & \WQ & \WU & \TK \\
.15 & \WK & \TK & \WK & \WK & \WQ & \TK & \TK & \WK \\
.16 & \WU & \WU & \WU & \WU & \WU & \WQ & \WU & \WU \\
.17 & \WQ & \WQ & \WU & \WQ & \WU & \WQ & \WQ & \WQ \\
.20 & \TK & \TK & \TK & \TK & \WU & \WU & \WU & \WU \\
.21 & \TK & \TK & \TK & \TK & \WU & \WU & \WU & \WU \\
.22 & \TK & \TK & \TK & \TK & \WU & \WU & \WQ & \WQ \\
.23 & \TK & \TK & \TK & \TK & \WU & \WU & \WQ & \WQ \\
.24 & \TK & \WK & \TK & \WK & \WU & \WU & \WU & \WU \\
.25 & \TK & \WK & \TK & \WK & \WU & \WU & \WU & \WU \\
.26 & \WQ & \WQ & \WQ & \WQ & \WU & \WU & \WU & \WU \\
.27 & \WQ & \WQ & \WQ & \WQ & \WU & \WU & \WU & \WU \\
.30 & \TK & \TK & \TK & \TK & \TK & \TK & \TK & \TK \\
.31 & \TK & \TK & \TK & \TK & \WU & \WK & \WQ & \WK \\
.32 & \TK & \TK & \TK & \TK & \WU & \WU & \WQ & \WQ \\
.33 & \TK & \TK & \TK & \TK & \WU & \WQ & \WQ & \WQ \\
.34 & \WK & \WK & \WU & \WU & \WU & \WU & \WU & \WU \\
.35 & \WQ & \WK & \TK & \TK & \WU & \WQ & \WQ & \WQ \\
.36 & \WU & \WU & \WU & \WU & \WU & \WU & \WU & \WU \\
.37 & \WU & \WU & \TK & \TK & \WU & \WQ & \WQ & \WU \\
.40 & \WQ & \WQ & \WQ & \WQ & \WU & \WU & \WU & \WU \\
.41 & \WQ & \WQ & \WQ & \WQ & \WU & \WU & \WU & \WU \\
.42 & \WQ & \WQ & \WQ & \WQ & \WU & \WU & \WU & \WU \\
.43 & \WQ & \WQ & \WQ & \WQ & \WU & \WU & \WU & \WU
\end{tabular}
&
\begin{tabular}{>{\bfseries}c|c@{\hspace{0.15cm}}c@{\hspace{0.15cm}}c@{\hspace{0.15cm}}c@{\hspace{0.15cm}}c@{\hspace{0.15cm}}c@{\hspace{0.15cm}}c@{\hspace{0.15cm}}c}
 & \multicolumn{8}{c}{$\mathbf{d_3}$} \\
$\mathbf{d_1d_2}$ & \textbf{0} & \textbf{1} & \textbf{2} & \textbf{3} & \textbf{4} & \textbf{5} & \textbf{6} & \textbf{7} \\ \hline
.44 & \WK & \WK & \WK & \WK & \WU & \WU & \WU & \WU \\
.45 & \WQ & \WQ & \WQ & \WQ & \WQ & \WQ & \WQ & \WQ \\
.46 & \WK & \WK & \WK & \WK & \WU & \WU & \WU & \WU \\
.47 & \WQ & \WQ & \WQ & \WQ & \WQ & \WQ & \WQ & \WQ \\
.50 & \TK & \TK & \TK & \TK & \TK & \TK & \TK & \TK \\
.51 & \TK & \TK & \WK & \WK & \TK & \TK & \WQ & \WQ \\
.52 & \TK & \TK & \TK & \TK & \WQ & \WQ & \WQ & \WQ \\
.53 & \TK & \TK & \TK & \TK & \WQ & \WQ & \WQ & \WQ \\
.54 & \WQ & \WQ & \WQ & \WQ & \WQ & \WQ & \WQ & \WQ \\
.55 & \TK & \TK & \TK & \TK & \TK & \TK & \WQ & \WQ \\
.56 & \TK & \TK & \TK & \TK & \WU & \WU & \WU & \WU \\
.57 & \WK & \WK & \TK & \TK & \WQ & \WQ & \WQ & \WQ \\
.60 & \WU & \WU & \TK & \TK & \WU & \WU & \WU & \WU \\
.61 & \WU & \WU & \TK & \TK & \WU & \WU & \WU & \WU \\
.62 & \WU & \WU & \TK & \TK & \WU & \WU & \WU & \WU \\
.63 & \WU & \WU & \TK & \TK & \WU & \WU & \WU & \WU \\
.64 & \WU & \WU & \WU & \WU & \WK & \WK & \WK & \WK \\
.65 & \WU & \WU & \WU & \WU & \WK & \WK & \WK & \WK \\
.66 & \WU & \WU & \WU & \WU & \WK & \WK & \WK & \WK \\
.67 & \WU & \WU & \WU & \WU & \WK & \WK & \WK & \WK \\
.70 & \TK & \TK & \TK & \TK & \TK & \TK & \TK & \TK \\
.71 & \WK & \WK & \WK & \WK & \WQ & \WQ & \WK & \WK \\
.72 & \WK & \WK & \WK & \WK & \WK & \WK & \WK & \WK \\
.73 & \TK & \TK & \WQ & \WQ & \WQ & \WQ & \WQ & \WQ \\
.74 & \WU & \WU & \WU & \WU & \WU & \WU & \WU & \WU \\
.75 & \WK & \WK & \WK & \WK & \WQ & \WQ & \WQ & \WQ \\
.76 & \WU & \WU & \WU & \WU & \WU & \WU & \WU & \WU \\
.77 & \WK & \WK & \TK & \TK & \WU & \WU & \WU & \WU \\
4.0 & \TK & \TK & \TK & \TK & \TK & \TK & \WQ & \WQ \\
4.1 & \TK & \TK & \TK & \TK & \TK & \TK & \WQ & \WQ \\
4.2 & \TK & \TK & \TK & \TK & \TK & \TK & \WQ & \WQ \\
4.3 & \TK & \TK & \TK & \TK & \TK & \TK & \TK & \TK \\
4.4 & \WK & \WK & \WK & \WK & \WU & \WU & \WU & \WU \\
4.5 & \TK & \TK & \WQ & \WQ & \TK & \TK & \WK & \WK \\
4.6 & \WK & \WK & \WK & \WK & \WU & \WU & \WU & \WU \\
4.7 & \WK & \WK & \TK & \TK & \WK & \WK & \WQ & \WQ
\end{tabular}
\end{tabular}
}

\vspace{0.35cm}

\begin{tabular}{cp{10.5cm}}
\TK & Tame; solution is known in both normal and mis\`ere play. \\
\WK & Wild; solution is known in both normal and mis\`ere play.  See~\cite{siegel_200Xd} for details. \\
\WU & Wild; solution is unknown in either normal or mis\`ere play.  For details on two-digit octals, see~\cite{siegel_200Xd}; for three-digit octals, see~\cite{miseregames_www}. \\
\WQ & Wild; solution is known in normal play, but not in mis\`ere play.  These are the most promising candidates for further research.
\end{tabular}

\caption{Summary of 3-digit octal games.}
\label{figure:3digitsummary}
\end{figure}

\begin{sidewaysfigure}
\centering
\begin{tabular}{|>{\bfseries}l||>{$}c<{$}|r|>{$}l<{$}|}
\hline
\mdseries Code & \Q & p & \multicolumn{1}{c|}{$\Phi$} \\ \hline
0.34 & \mathcal{S}_{12} & 8 & \cpretend{a 1 a b 1 a 1 ab a c a b 1 ac 1 ab a c a b 1 ac} \\
0.75\=0 & \mathcal{R}_8 & 2 & \cpretend{a b a b c b c b ab2 b ab2 b ab2} \\
0.51\.2 & \mathcal{R}_8 & 6 & \cpretend{a a a b b b a b2 c b b b ab2 b2 c b b b ab2 b2 ab2 b b b ab2 b2 ab2} \\
0.71\.2 & \mathcal{R}_{14} & 6 & \cpretend{a b a 1 c 1 a ac a 1 c 1 a ac} \\
0.71\.6 & \mathcal{R}_{14} & 2 & \cpretend{a b a 1 c 1 c ac a ac c ac c ac c ac ac2 ac c c2 c c2 ac2 c2 ac2 c2 ac2} \\
4.5\.6 & \mathcal{R}_8 & 4 & \cpretend{a a b b a c b b c c b b ab2 ab2 b b ab2 ab2} \\
4.7\.4 & \mathcal{R}_8 & 2 & \cpretend{a b a b c b c b ab2 b ab2 b ab2} \\
0.3101 & \mathcal{R}_{14} & 2 & \cpretend{a b c2 c c2 ac2 c2 ac2 c2 ac2} \\
0.3131 & \mathcal{S}_{12} & 2 & \cpretend{a b a b b2 c b2 ab2 b2 ab2 b2 ab2} \\
\hline
\end{tabular}
\caption{Pretending functions for octal games with mis\`ere quotients isomorphic to $\mathcal{R}_8$, $\mathcal{S}_{12}$ or $\mathcal{R}_{14}$ (cf.~Figure~\ref{figure:day4quotients}).}
\label{figure:simpleoctals}

\vspace{0.5cm}

\begin{tabular}{|>{\bfseries}l||r|>{$}l<{$}|>{$}c<{$}|>{$}c<{$}|}
\hline
\mdseries Code & p & \multicolumn{1}{c|}{$\Phi$} & \Q & \P \\
\hline
0.123\.0 & 5 & \cpretend{a 1 b b a d2 1 c d a d2 1 c d} &
\raisebox{-8.25ex}[0pt][0pt]{\present{4.25cm}{a,b,c,d}{a2=1,b4=b2,b2c=b3,c2=1,b2d=d,cd=bd,d3=ad2}} &
\raisebox{-8.25ex}[0pt][0pt]{\centeredplist{1.5cm}{a,b2,ac,bd,d2}}
\\
0.1023 & 7 & \cpretend{a 1 1 b b a a d2 1 1 c d a a d2 1 1 c d} & & \\
0.1032 & 7 & \cpretend{a 1 a 1 b b b a d2 d2 1 c ad2 d a d2 d2 1 c ad2 d} & & \\
0.1033 & 7 & \cpretend{a 1 a b b b 1 d2 d2 a d ad2 c 1 d2 d2 a d ad2 c} & & \\
0.1331 & 5 & \cpretend{a a b b 1 d2 a d c 1 d2 a d c} & & \\
0.3103 & 5 & \cpretend{a b d2 d 1 a c d2 d 1 a c} & & \\
0.3112 & 5 & \cpretend{a b a 1 b d2 c a 1 d d2 c a 1 d} & & \\
\hline
& & & & \\
0.123\.2 & 6 & \cpretend{a 1 b b ab a b2 1 c b d a b2 1 c b d} &
\raisebox{-1.5ex}[0pt][0pt]{\present{5cm}{a,b,c,d}{a2=1,b4=b2,b2c=b3,bc3=bc,c4=c2,bcd=b2d,b3d2=bd2,cd2=bd2,b2d3=d3,d5=ad4}} &
\raisebox{-1.5ex}[0pt][0pt]{\centeredplist{2.15cm}{a,b2,ac,ac2,ac3,bd,ab3d,cd,ad2,b2d2,abd3,d4}} \\
0.1323 & 6 & \cpretend{a a 1 b b ab a b2 1 c b d a b2 1 c b d} & & \\
& & & & \\
\hline
\end{tabular}
\caption{Pretending functions and quotient presentations for some wild quaternary games.}
\label{figure:quaternaries}
\end{sidewaysfigure}

\begin{sidewaysfigure}
\centering
\begin{tabular}{|>{\bfseries}l||>{$}c<{$}|>{$}c<{$}|c|}
\hline
\mdseries Code & \Q & \P & \multicolumn{1}{c|}{$\Phi$} \\ \hline
0.53\.0 &
\present{6cm}{a,b,c,d,e}{a2=1,b3=b,b2c=c,c2=b2,bd=ab,cd=ac,d2=b2,be=ab,ce=ac,e2=b2} &
\plist{3cm}{a,b2,e} &
\cxpretend{9}{a a b b a b2 b b c b2 d b b e ab2 b b c ab2 ade b b ab2 ab2 b b c ab2 ab2 b b ab2 ab2 b b c ab2 ab2}
\\ \hline
0.71\.0 &
\present{6cm}{a,b,c,d}{a2=1,b4=b2,b2c=c,c4=ac3,c3d=c3,d2=1} &
\plist{3cm}{a,b2,bc,c2,ac3,ad,b3d,cd,bc2d} &
\cxpretend{6}{a b a 1 c 1 a d a 1 c 1 a d}
\\ \hline
0.72\=0 &
\present{6cm}{a,b,c,d,e,f}{a2=1,b4=b2,b2c=ab2,c2=b2,b2d=b3,cd=ab3,d2=b2,b2e=ab3,bce=bd,de=abd,e2=b2,bf=bd,cf=ab3,df=b2,ef=ab2,f2=b2} &
\plist{3cm}{a,b2,ad,be,af} &
\cxpretend{4}{a 1 b ab a 1 b ab c b2 abc ab3 c b2 d e abd b2 f e abd b2 f}
\\ \hline
0.144 &
\present{6cm}{a,b,c,d,e}{a2=1,b3=b,b2c=c,c4=c2,bd=ab,cd=ac,d2=b2,c2e=c3,de=ab2e,e2=c2} &
\plist{3cm}{a,b2,c2,ae,abce} &
\cxpretend{10}{a 1 1 a b b b b c c a ab2 ab2 d b b b b c3 e ab2 ab2 ab2 d b b b b c3 e ab2}
\\ \hline
0.2\.4\"1 &
\multicolumn{2}{c|}{$\cong \Q(\textbf{0.71})$} &
\cxpretend{10}{1 a b a 1 a 1 c 1 a 1 a d a 1 a 1 c 1 a 1 a d}
\\ \hline
0.351 &
\present{6cm}{a,b,c,d}{a2=1,b3=b,b2c=c,c3=ac2,b2d=d,c2d=ac2,d2=b2} &
\plist{3cm}{a,b2,bc,c2,d,bcd} &
\cxpretend{8}{a b a b b2 c b2 b d b d b b2 c b2 b d b d}
\\ \hline
\end{tabular}
\caption{Quotient presentations and pretending functions for \textbf{0.53}, \textbf{0.71}, \textbf{0.72}, \textbf{0.144}, \textbf{0.241}, and \textbf{0.351}.}
\label{figure:053}
\label{figure:071}
\label{figure:072}
\label{figure:0144}
\label{figure:0241}
\label{figure:0351}
\end{sidewaysfigure}

\begin{figure}
\centering
\[\begin{array}{@{}r@{~}l@{}}
\Q(\textbf{0.115}) \cong & \present[t]{9.5cm}{a,b,c,d,e,f,g,h,i}{a2=1,b4=b2,bc=ab3,c2=b2,b2d=d,cd=ad,d3=ad2,b2e=b3,de=bd,be2=ace,ce2=abe,e4=e2,bf=b3,df=d,ef=ace,cf2=cf,f3=f2,b2g=b3,cg=ab3,dg=bd,eg=be,fg=b3,g2=bg,bh=bg,ch=ab3,dh=bd,eh=bg,fh=b3,gh=bg,h2=b2,bi=bg,ci=ab3,di=bd,ei=be,fi=b3,gi=bg,hi=b2,i2=b2} \vspace{0.25cm} \\
\P = & \plist[t]{9.5cm}{a,b2,bd,d2,ae,ae2,ae3,af,af2,ag,ah,ai}
\end{array}\]

\xpretend{14}{a a a 1 a a a b b b a b b b a a a 1 c c c b b b d b b e c c c f c c c b b g d h h i ab2 ab2 abg f abg abg abe b3 b3 h d h h h ab2 ab2 abe f2 abg abg abg b3 b3 h d h h h ab2 ab2 abg f2 abg abg abg b3 b3 b3 d b3 b3 b3 ab2 ab2 abg f2 abg abg abg b3 b3 b3 d b3 b3 b3 ab2 ab2 ab2 f2 ab2 ab2 ab2 b3 b3 b3 d b3 b3 b3 ab2 ab2 ab2 f2 ab2 ab2 ab2}

\caption{Quotient presentation and pretending function for \textbf{0.115}.}
\label{figure:0115}
\end{figure}

\begin{figure}
\centering
\[\begin{array}{@{}r@{~}l@{}}
\Q(\textbf{0.152}) \cong &
\hangingpresent[t]{8cm}{a,b,c,d,e}{a2=1,b3=b,bc2=b,c3=c,bd=b,c2d=d,d2=b2,b2e=e,c2e=e,de=e,e3=ce2}
\vspace{0.25cm} \\
\P = & \plist[t]{9cm}{a,b2,c2,acd,abce,e2}
\end{array}\]

\xpretend{12}{a a 1 b b b 1 a b2 c ab b ab ab2 b2 a ab b b ce2 d ab2 b2 e b b ab ab2 b2 ab2 ab b ab ce2 b2 ab2 b2 b b b b2 ab2 b2 ce2 ab b ab ce2 b2 ab2 b2 b b b b2 ab2 b2 ce2 ab b ab ab2 b2 ab2 ab b b ce2 b2 ab2 b2 ce2 b b ab ab2}

\caption{Quotient presentation and pretending function for \textbf{0.152}.}
\label{figure:0152}
\end{figure}

\begin{figure}
\centering
\[\begin{array}{@{}r@{~}l@{}}
\Q(\textbf{0.77}) \cong &
\hangingpresent[t]{7.5cm}{a,b,c,d,e,f,g}{a2=1,b3=b,bc2=b,c3=c,bd=bc,cd=b2,d3=d,be=bc,ce=b2,e2=de,bf=ab,cf=ab2c,d2f=f,f2=b2,b2g=g,c2g=g,dg=cg,eg=cg,fg=ag,g2=b2}
\vspace{0.25cm} \\
\P = & \plist[t]{10cm}{a,b2,ac,ac2,d,ad2,e,ade,adf}
\end{array}\]

\xpretend{12}{a b ab a c ab b ab2 d b bc e ab2 b abc ab2 d2e ab b ade b2c bc abc b2c f b g ab2c b2c abc b ab2 g bc abc b2c ab2 b ab ab2 b2c abc b ab2 g b abc b2c ab2 b g ab2 b2c abc b ab2 b2c b abc b2c ab2 b g ab2 b2c abc b ab2 g bc abc b2c ab2 b g ab2 b2c abc b ab2 g b abc b2c ab2 b g ab2 b2c abc b ab2 g b}

\caption{Quotient presentation and pretending function for \textbf{0.77}.}
\label{figure:077}
\end{figure}

\begin{figure}
\centering
\begin{tabular}{l@{}c}
\begin{minipage}{6.75cm}
$\begin{array}{c@{~}l}
\mathcal{Q} \cong \langle a,b,c_n~| &
{a^2=1},\ {b^{n+1}c_n=b^{2n+3}},\\
& {(c_mc_n=b^{m+2}c_n)_{m \leq n}} \rangle
\end{array}$

\vspace{0.5cm}

$\mathcal{P} = \{a,(b^{2m})_{m \geq 1},(b^mc_n)_{m \leq n\ \textrm{and}\ m+n\ \textrm{odd}}\}$
\end{minipage}
&
\pretend{4}{1 a b ab 1 a b ab c_0 ac_0 c_1 ab^3 c_2 abc_1 c_3 abc_2 c_4 abc_3 c_5 abc_4 c_6 abc_5}
\end{tabular}

\caption{\label{figure:026}Quotient presentation and pretending function for \textbf{0.26}.}
\end{figure}

\begin{figure}
\centering
\begin{tabular}{l@{}c}
\begin{minipage}{8.25cm}
\centering
$\begin{array}{c@{~}l}
\mathcal{Q} \cong \langle a,b,c,d_n~| &
{a^2=1},\ {bc=ab^3},\ {c^2=b^4}, \\
& {b^{n+1}d_n=a^{n+1}b^{2n+5}}, {cd_n=ab^2d_n}, \\
& {d_md_n=a^{m+1}b^{m+4}d_n} \rangle
\end{array}$

\vspace{0.5cm}

$\begin{array}{c@{}l}
\P = \{ & a,(b^{2m})_{m \geq 1},(b^md_n)_{m\ \textrm{odd},\ n\ \textrm{even},\ m < n},\\
& (ab^md_n)_{m\ \textrm{even},\ n\ \textrm{odd},\ m < n}\}
\end{array}$

\end{minipage}
&
\pretend{2}{a b a b c b^3 d_0 d_1 d_2 d_3 d_4 d_5 d_6}
\end{tabular}

\caption{\label{figure:47}Quotient presentation and pretending function for \textbf{4.7}.}
\end{figure}

\clearpage

\subsection*{Algebraic Periodic Games: \textbf{0.26} and \textbf{4.7}}
\suppressfloats[t]
\label{algebraicperiodicity}
In Section A.5 we gave two examples of algebraic periodic games: \textbf{0.26} and \textbf{4.7} have infinite (and non-finitely generated) mis\`ere quotients.  Furthermore, all of their \emph{partial} quotients are finite.  Full presentations are shown in Figures~\ref{figure:026} and~\ref{figure:47}, respectively.


Since we do not have any computational methods for verifying algebraic periodicity, we must resort to manual proofs of Figures~\ref{figure:026} and~\ref{figure:47}.  In each case, the proof proceeds in two stages.  We first show that the given presentation is \emph{correct}, in the sense that $G$ is a \PPos{} iff $\Phi(G) \in \P$.  Then we show that the given presentation is \emph{reduced} (as a bipartite monoid).

\subsection*{Figure~\ref{figure:026} Is Correct for \textbf{0.26}}


Let $\A$ be the set of $\mathbf{0.26}$ positions, regarded as a free commutative monoid on the heap alphabet $\mathscr{H} = \{H_1,H_2,H_3,\ldots\}$.  Define homomorphisms $t,w : \A \to \mathbb{N}$ as follows:
\[t(H_k) = \begin{cases}
0 & \textrm{if } k \in \{1,2\}; \\
1 & \textrm{if } k \in \{3,4\}; \\
\lfloor \frac{k-5}{2} \rfloor & \textrm{if } k \geq 5.
\end{cases}\qquad
w(H_k) = \begin{cases}
0 & \textrm{if } k \leq 8; \\
1 & \textrm{if } k \in \{9,10\}; \\
(k-7)/2 & \textrm{if $k \geq 11$ is odd}; \\
(k-12)/2 & \textrm{if $k \geq 11$ is even}.
\end{cases}
\]

Let $g(G)$ denote the ordinary Grundy value of $G$.  It is easily checked that $g(H_k)$ is equal to the mod-4 parity of $k-1$.




\begin{fact}[Allemang~\cite{allemang_1984}]
\label{fact:ppos026}
Let $G \in \A$ and write $G = X + H$, where $H = H_k$ is the largest single heap appearing in $G$.  Then $G$ is a \PPos{} iff one of the following conditions holds:
\begin{enumerate}
\item[(i)] $t(G) = 0$ and $g(G) = 1$;
\item[(ii)] $t(G) \neq 0$, $w(H) \leq t(X)$, and $g(G) = 0$; or
\item[(iii)] $t(G) \neq 0$, $w(H) \geq t(X) + 2$, and $g(G) = 2$.
\end{enumerate}
\end{fact}

Fact~\ref{fact:ppos026} corrects some slight errors in Allemang's definition of~$t$.  This would be a good reason to include a proof here; but the only proof we know is a tedious and unenlightening combinatorial slog, and anyway it is our hope that algorithmic verification methods will emerge in the near future.  We therefore choose to take the easy way out, and leave its proof as an exercise.

\begin{theorem}
Figure~\ref{figure:026} is correct for \textbf{0.26}.
\end{theorem}

\begin{proof}
Let $G \in \A$.  We must show that $G$ is a \PPos{} iff $\Phi(G) \in \P$.  We first make two preliminary observations, which are easily proved by inspecting the definitions of $g$, $t$, and $\Phi$:
\[\label{eq:twoheaps} \textrm{If $k \leq k'$, then } \Phi(H_k)\Phi(H_{k'}) = a^{g(H_k)}b^{t(H_k)}\Phi(H_{k'}). \tag{\dag}\]
\[\label{eq:severalheaps} \textrm{If $w(H) \leq m$, then } b^m\Phi(H) = a^{g(H)}b^{m+t(H)}. \tag{\ddag}\]

Now fix $G$, and write $G = X + H$ where $H = H_k$ is the largest heap appearing in $G$.  By repeated application of (\ref{eq:twoheaps}), we have $\Phi(G) = a^{g(X)}b^{t(X)}\Phi(H)$.  There are five cases.

\vspace{0.15cm}\noindent
\emph{Case 1}: $t(G) = 0$.  Then $\Phi(G) = a^{g(G)}$ and the conclusion is evident.

\vspace{0.15cm}\noindent
\emph{case 2}: $t(G) \neq 0$ and $w(H) \leq t(X)$.  Then by (\ref{eq:severalheaps}) we have
\[\Phi(G) = a^{g(X)}b^{t(X)}\Phi(H) = a^{g(X)}a^{g(H)}b^{t(X)+t(H)} = a^{g(G)}b^{t(G)}.\]
So $\Phi(G) \in \P$ iff $g(G)$ and $t(G)$ are both even.  But $t(G)$ is even iff $g(G) \in \{0,1\}$, so it follows that
\[\Phi(G) \in \P \textrm{ iff } g(G) = 0.\]
This agrees with the characterization in Fact \ref{fact:ppos026}.

\vspace{0.15cm}\noindent
\emph{Case 3}: $t(G) \neq 0$, $w(H) > t(X)$, and $k \in \{9,10\}$.  Then $w(H) = 1$, so $t(X) = 0$, and we have $\Phi(G) = a^{g(X)}\Phi(H)$.  Since $\Phi(H_9) = c_0$ and $\Phi(H_{10}) = ac_0$, this necessarily implies $\Phi(G) \not\in \P$, which agrees with Fact~\ref{fact:ppos026}.

\vspace{0.15cm}\noindent
\emph{Case 4}: $t(G) \neq 0$, $w(H) > t(X)$, and $k \geq 11$ is odd.  Then $\Phi(H) = c_{w(H)-1}$ and we have
\[\Phi(G) = a^{g(G)}b^{t(X)}c_{w(H)-1}.\]
Thus $\Phi(G) \in \P$ precisely when $g(G)$ and $t(X) + w(H)$ are both even.  But $t(X) + w(H)$ is even iff $t(G)$ is odd, so this agrees with Fact~\ref{fact:ppos026}.

\vspace{0.15cm}\noindent
\emph{Case 5}: $t(G) \neq 0$, $w(H) > t(X)$, and $k \geq 11$ is even.  Then $\Phi(H) = abc_{w(H)}$ and so
\[\Phi(G) = a^{g(G)}b^{t(X)+1}c_{w(H)}.\]
Thus $\Phi(G) \in \P$ precisely when $g(G)$ and $t(X) + w(H)$ are both even, and the conclusion is just as in Case 4.
\end{proof}

This shows that $\Phi : \A \to \Q$ is a homomorphism of bipartite monoids.  The following theorem completes the picture.

\begin{theorem}
The bipartite monoid $\QP$ given in Figure \ref{figure:026} is reduced.
\label{theorem:026reduced}
\end{theorem}

\begin{proof}
For convenience, write $c_\infty = 1$.  Then every $x \in \Q$ may be written uniquely as $a^ib^mc_n$, where $i \in \{0,1\}$, $m \in \mathbb{N}$, $n \in \mathbb{N} \cup \{\infty\}$, and $m \leq n$.  We will prove that, if $(i,m,n) \neq (i',m',n')$, then $x = a^ib^mc_n$ and $x' = a^{i'}b^{m'}c_{n'}$ are distinguishable.

First suppose $i \neq i'$, so without loss of generality $i = 0$ and $i' = 1$.  Fix an integer $M > \max\{n,n'\}$ such that $m + n + M$ is even.  Then $xb^M = b^{m+n+2+M} \in \P$, while $x'b^M = ab^{m'+n'+2+M} \not\in \P$.

Now assume $i = i'$.  It suffices to assume $i = 0$, since if $x$ and $x'$ are distinguished by $z$, then $ax$ and $ax'$ are distinguished by $az$.  We may also assume, without loss of generality, that $n-m \leq n'-m'$.  There are several cases.

\vspace{0.15cm}\noindent
\emph{Case 1}: $m + n \not\equiv m' + n' \modx{2}$.  Then let $M > \max\{n,n'\}$ be such that $m + n + M$ is even.  As before, $(b^mc_n)b^M = b^{m+n+2+M}$, while $(b^{m'}c_n)b^M = b^{m'+n'+2+M}$.  Since $m+n+2+M$ is even and $m'+n'+2+M$ is odd, we are done.

\vspace{0.15cm}\noindent
\emph{Case 2}: $m + n \equiv m' + n' \modx{2}$ and $n-m < n'-m'$.  Without loss of generality, we may assume that $n-m < n'-m'$.  Then $xb^{n-m+1} = b^{n+1}c_n = b^{2n+3} \not\in \P$, while $x'b^{n-m+1} = b^{m'+n-m+1}c_{n'}$.  But $m'+n-m+1 < m'+n'-m'+1 = n'+1$, and
\[m'+n-m+1+n' \equiv (m'+n') + (n-m) + 1 \equiv (m'+n') + (m+n) + 1 \equiv 1 \modx{2},\]
so this is in $\P$.

\vspace{0.15cm}\noindent
\emph{Case 3}: $n-m = n'-m'$.  Then we may assume that $n < n'$.  Let $M > \max\{n,n'\}$ and put $N = m'+n'+2+M$.  Then $x'b^Mc_N = b^{m'+n'+2+M}c_N \not\in \P$.  However,
\[xb^Mc_N = b^{m+n+2+M}c_N = b^{m+n+4+M+N}.\]
Since $m+n \equiv m'+n' \modx{2}$, this exponent is even.  Hence $xb^Mc_N \in \P$.
\end{proof}

\subsection*{Figure~\ref{figure:47} Is Correct for \textbf{4.7}}

Let $\A$ be the set of \textbf{4.7} positions, regarded as a free commutative monoid on the heap alphabet $\mathscr{H} = \{H_1,H_2,H_3,\ldots\}$.  Define homomorphisms $t,w : \A \to \mathbb{N}$ as follows:

\[t(H_k) = \begin{cases}
0 & \textrm{if } k = 1; \\
1 & \textrm{if } k = 2; \\
k-3 & \textrm{if } k \geq 3.
\end{cases}\qquad
w(H_k) = \begin{cases}
0 & \textrm{if } k \leq 4; \\
1 & \textrm{if } k = 5; \\
k-6 & \textrm{if } k \geq 6.
\end{cases}
\]

Let $g(G)$ denote the ordinary Grundy value of $G$.  It is easily checked that $g(H_k) = 2$ if $k$ is even, $1$ if $k$ is odd.

\begin{fact}[Allemang~\cite{allemang_1984}]
\label{fact:ppos47}
Let $G \in \A$ and write $G = X + H$, where $H = H_k$ is the largest single heap appearing in $G$.  Then $G$ is a \PPos{} iff one of the following conditions holds:
\begin{enumerate}
\item[(i)] $t(G) = 0$ and $g(G) = 1$;
\item[(ii)] $t(G) \neq 0$, $w(H) \leq t(X)$, and $g(G) = 0$; or
\item[(iii)]$t(G) \neq 0$, $w(H) \geq t(X) + 2$, and $g(G) = 3$.
\end{enumerate}
\end{fact}

\begin{theorem}
Figure~\ref{figure:47} is correct for \textbf{4.7}.
\end{theorem}

\begin{proof}
Fix $G \in \A$.  We must show that $G$ is a \PPos{} iff $\Phi(G) \in \P$.  As in the case of \textbf{0.26}, the following observations are easily checked:

\[\label{eq:twoheaps47} \textrm{If $k \leq k'$, then } \Phi(H_k)\Phi(H_{k'}) = a^{g(H_k)}b^{t(H_k)}\Phi(H_{k'}). \tag{\dag}\]
\[\label{eq:severalheaps47} \textrm{If $w(H) \leq m$, then } b^m\Phi(H) = a^{g(H)}b^{m+t(H)}. \tag{\ddag}\]

Now fix $G$, and write $G = X + H$ where $H = H_k$ is the largest heap appearing in $G$.  By repeated application of (\ref{eq:twoheaps47}), we have $\Phi(G) = a^{g(X)}b^{t(X)}\Phi(H)$.  There are four cases.

\vspace{0.15cm}\noindent\emph{Case 1}:
$t(G) = 0$.  Then $\Phi(G) = a^{g(G)}$ and the conclusion is evident.

\vspace{0.15cm}\noindent\emph{Case 2}:
$t(G) \neq 0$ and $w(H) \leq t(X)$.  Then by (\ref{eq:severalheaps47}) we have
\[\Phi(G) = a^{g(X)}b^{t(X)}\Phi(H) = a^{g(X)}a^{g(H)}b^{t(X)+t(H)} = a^{g(G)}b^{t(G)}.\]
So $\Phi(G) \in \P$ iff $g(G)$ and $t(G)$ are both even.  But $t(G)$ is even iff $g(G) \in \{0,1\}$, so this agrees with the characterization in Fact~\ref{fact:ppos47}.

\vspace{0.15cm}\noindent\emph{Case 3}:
$t(G) \neq 0$, $w(H) > t(X)$, and $k \in \{5,6\}$.  Then $w(H) = 1$, so $t(X) = 0$, and we have $\Phi(G) = a^{g(X)}\Phi(H)$.  Since $\Phi(H_5) = c$ and $\Phi(H_6) = b^3$, this necessarily implies $\Phi(G) \not\in \P$, which agrees with Fact~\ref{fact:ppos47}.

\vspace{0.15cm}\noindent\emph{Case 4}:
$t(G) \neq 0$, $w(H) > t(X)$, and $k \geq 7$.  Then $\Phi(H) = d_{w(H)-1}$, so
\[\Phi(G) = a^{g(X)}b^{t(X)}d_{w(H)-1}.\]

Now if $w(H)$ is odd, then $\Phi(G) \in \P$ iff $g(X)$ is even and $t(X)$ is odd.  But $w(H) \equiv g(H) \pmod{2}$, and $w(H) \not\equiv t(H) \pmod{2}$, so this means $g(G)$ and $t(G)$ are both \emph{odd}.  Since $t(G)$ is odd iff $g(G) \in \{2,3\}$, this agrees with the characterization in Fact~\ref{fact:ppos47}.

Likewise, if $w(H)$ is even, then $\Phi(G) \in \P$ iff $g(X)$ is odd and $t(X)$ is even.  Once again, this means that $g(G)$ and $t(G)$ are both odd.  This exhausts all cases and completes the proof.
\end{proof}

\begin{theorem}
The bipartite monoid $\QP$ given in Figure~\ref{figure:47} is reduced.
\label{theorem:47reduced}
\end{theorem}

The proof of Theorem~\ref{theorem:47reduced} is much like Theorem~\ref{theorem:026reduced}.

\subsection*{Games Born by Day 4}
\suppressfloats[t]
\label{dayfour}
The quotients $\Q(G)$, for a single mis\`ere game $G$, are fundamental (but in some cases still quite intricate).
Figure~\ref{figure:day4quotients} summarizes all quotients obtained this way, to birthday four.  Each quotient is listed together with its order, monoid presentation, and $\SP$-portion.



The canonical forms of the twenty-two mis\`ere impartial games $G$ of birthday at most four were introduced by Conway~\cite{conway_1976}, and are duplicated here in Figure~\ref{figure:day4games}.  They are shown together with their mis\`ere quotients.

\begin{figure}[tb]
\centering
$\begin{array}{|c|r|c|c|}
\cline{2-4}
\multicolumn{1}{c|}{} & |\Q| & \Q & \P \bigstrut \\ \hline
\mathcal{T}_0 & 1 & \<~|~\> & \emptyset \bigstrut \\
\mathcal{T}_1 & 2 & \<a~|~a^2=1\> & \{a\} \bigstrut \\
\mathcal{T}_2 & 6 & \<a,b~|~a^2=1,\ b^3=b\> & \{a,b^2\} \bigstrut \\ 
\mathcal{R}_8 & 8 & \<a,b,c~|~a^2=1,\ b^3=b,\ bc=ab,\ c^2=b^2\> & \{a,b^2\} \bigstrut \\
\mathcal{T}_3 & 10 & \<a,b,c~|~a^2=1,\ b^3=b,\ c^3=c,\ c^2=b^2\> & \{a,b^2\} \bigstrut \\ 
\mathcal{S}_{12} & 12 & \<a,b,c~|~a^2=1,\ b^4=b^2,\ b^2c=b^3,\ c^2=1\> & \{a,b^2,ac\} \bigstrut \\ 
\mathcal{S}_{12}' & 12 & \<a,b,c~|~a^2=1,\ b^3=b,\ c^2=1\> & \{a,b^2,ac\} \bigstrut \\
\mathcal{R}_{14} & 14 & \<a,b,c~|~a^2=1,\ b^3=b,\ b^2c=c,\ c^3=ac^2\> & \{a,b^2,bc,c^2\} \bigstrut \\ 
\hline
\end{array}$
\caption{The eight mis\`ere quotients born by day 4.}
\label{figure:day4quotients}
\end{figure}

\begin{figure}
\centering
$\begin{array}{@{}|@{\ }c@{\ }|@{\ }c@{\ }|@{\ }c@{\ }|@{\ }c@{\ }|@{\ }c@{\ }|@{\ }c@{\ }|@{\ }c@{\ }|@{\ }c@{\ }|@{}}
\hline
0 &
\Star &
\begin{array}{@{}ccc@{}} \Star2 & \Star3 & \Star2_\sh \\ \Star3_\sh & \Star32 & \Star2_{\sh\sh} \\ \Star2_\sh2 & \Star2_\sh210 & \Star2_\sh3 \\ & \Star2_\sh32 \end{array} &
\Star2_\sh320 &
\begin{array}{@{}c@{}} \Star4 \\ \Star2_\sh3210 \end{array} &
\begin{array}{@{}c@{}} \Star2_\sh1 \\ \Star2_\sh21 \\ \Star2_\sh31 \end{array} &
\Star2_\sh321 &
\begin{array}{@{}c@{}} \Star2_\sh0 \\ \Star2_\sh20 \\ \Star2_\sh30 \end{array}
\\ \hline
\mathcal{T}_0 & \mathcal{T}_1 & \mathcal{T}_2 & \mathcal{R}_8 & \mathcal{T}_3 & \mathcal{S}_{12} & \mathcal{S}_{12}' & \mathcal{R}_{14} \bigstrut \\ \hline
\end{array}$
\caption{The 22 mis\`ere games born by day 4, grouped by quotient isomorphism type.}
\label{figure:day4games}
\end{figure}

There are exactly 4171780 games born by day five~\cite{conway_1976}.  They yield a bewildering variety of mis\`ere quotients, including infinite quotients; finite quotients with more than 1500 elements (and probably larger ones as well); and counterexamples to several reasonable-sounding statements about the general structure of mis\`ere quotients. It is conceivable that a complete survey of games born by day five might eventually be conducted.

\subsection*{The Games ${\bf 0.{(3310)}^n}$}
Suppose mis\`ere Nim is played with two restrictions:
\begin{itemize}
\item $4k$ beans may not be removed from any heap;
\item $4k+3$ beans may be removed only if it is the whole heap.
\end{itemize}

In standard notation, this game would be
\[\mathbf{0.(3310)^\infty} = \mathbf{0.3310331033103310\ldots}\]
Does its mis\`ere quotient have order~226?  The following approximations seem to suggest that it does:

\begin{center}
\begin{tabular}{lrr}
\multicolumn{1}{c}{$\Gamma$} & $|\Q|$ & pd \\ \hline
{\bf 0.3310}                 & 6                     & 3  \\ 
{\bf 0.33103310}             & 202                   & 7    \\ 
{\bf 0.331033103310}         & 226                   & 11      \\
{\bf 0.3310331033103310}     & 226                   & 15 \\
{\bf 0.33103310331033103310} & 226                   & 19 \\
\end{tabular}
\end{center}

\begin{question}
Does $\Q(\mathbf{0.(3310)^\infty}) = \Q(\mathbf{0.(3310)^n})$ for all $n \geq 3$?  Can we say anything interesting in general about games with infinite octal codes?
\end{question}

\section{Algorithms}

\label{appendix:algorithms}

In this appendix we describe \emph{MisereSolver}'s algorithm for calculating mis\`ere quotients.  Throughout this section, let $\A$ be an arbitrary finitely-generated closed set of games, treated as a free commutative monoid on generators $\mathscr{H} = \{H_1,H_2,\ldots,H_n\}$.  Denote by $\llex$ the lexicographic ordering on $\A$:
\[\left(\sum k_iH_i\right) \llex \left(\sum \ell_iH_i\right) \textrm{ iff $k_i < \ell_i$ for the largest $i$ such that $k_i \neq \ell_i$.}\]
Notice that if $X \in \A$ and $X'$ is an option of $X$, then $X' \llex X$.

The basic idea is to produce a sequence of increasingly accurate approximations to the mis\`ere quotient $\Q(\A)$.  If $\Q(\A)$ is indeed finite, then the sequence is guaranteed to converge to it.

At each stage of the iteration, we are given a promising r.b.m.\ $\QP$ (the ``candidate quotient''), together with a homomorphism $\Phi : \A \to \Q$ (represented as a mapping $\mathscr{H} \to \Q$).  If $\QP$ is the correct quotient of $\A$, with pretending function $\Phi$, then we are done.  If not, then there is some $X \in \A$, $X \neq 0$, such that either:
\begin{enumerate}
\item[(i)] $\Phi(X) \in \P$, but also $\Phi(X') \in \P$ for some option $X'$; or
\item[(ii)] $\Phi(X) \not\in \P$, but there is no option $X'$ with $\Phi(X') \in \P$.
\end{enumerate}
We say that such an $X$ is a \emph{failure} of the candidate $(\Q,\P,\Phi)$.  Now suppose that~$X$ is the \emph{lexicographically least} failure of $(\Q,\P,\Phi)$.  Then, by a straightforward induction, $\Phi$ correctly predicts the outcomes of all proper followers of $X$.  Therefore we know the outcome of $X$: it is necessarily the opposite of that predicted by $\Phi$.  Using this information, we produce a new candidate $(\Q',\P',\Phi')$, for which $X$ is \emph{not} a failure.  The construction will guarantee that no lexicographically smaller failures are introduced.  The least failures of successive candidate quotients are therefore strictly lexicographically increasing.

There are two main components of this algorithm.
\begin{itemize}
\item \emph{Verification}: Given a candidate $(\Q,\P,\Phi)$, the verification engine determines whether it is the correct quotient of $\A$, and efficiently identifies the least failure when it is not.
\item \emph{Recalibration}: Given a candidate $(\Q,\P,\Phi)$, together with the least failure~$X$, the recalibration engine constructs the next candidate $(\Q',\P',\Phi')$.
\end{itemize}

When \emph{MisereSolver} begins, it starts with an initial round of recalibration.  For the initial candidate, we use the mis\`ere quotient of $\<H_1,\ldots,H_{n-1}\>$ (computed recursively), taking $X = H_n$ to be the least failure.

We are reasonably certain that \emph{MisereSolver} uses the ``book'' algorithm for verification: it is both elegant and fast.  (In particular, we have improved substantially on the algorithms described in~\cite{plambeck_2005}.)  However, our recalibration algorithm feels very crude: \emph{MisereSolver} sometimes chooses very poor candidates; often they are considerably larger than the true quotient.  An improved recalibrator might dramatically extend \emph{MisereSolver}'s scope.

\subsection*{Verification}

Given a r.b.m.\ $\QP$ and a monoid homomorphism $\Phi : \A \to \Q$, we wish to determine:

\begin{enumerate}
\item[(a)] The $\llex$-least $X \in \A$ such that $\Phi(X) \in \P$ but $\Phi(X') \in \P$ for some option $X'$;
\item[(b)] The $\llex$-least $X \in \A$ such that $\Phi(X) \not\in \P$ but $\Phi(X') \not\in \P$ for all options~$X'$.
\end{enumerate}

As observed in \cite{plambeck_2005}, (b) is far more difficult computationally.  This is due to the presence of the universal quantifier: in searching for \emph{some} option in (a), we can essentially treat each generator as an independent entity; but to iterate over \emph{all} options in (b), we must consider the totality of generators involved in~$X$.  

For (a), we use the essential strategy outlined in \cite{plambeck_2005}.  The basis is the following theorem:

\begin{fact}[$\mathscr{P}$-Verification Theorem {\cite[Section 9.1.3]{plambeck_2005}}]
\label{fact_P_ver}
Let $\QP$ be a r.b.m.\ and fix a homomorphism $\Phi : \A \to \Q$.  The following are equivalent.
\begin{enumerate}
\item[(i)] There is a game $X \in \A$ and an option $X'$ such that $\Phi(X),\Phi(X') \in \P$.
\item[(ii)] There is a generator $H \in \mathscr{H}$, an option $H'$, and an element $x \in \Q$ such that $x\Phi(H),x\Phi(H') \in \P$.
\end{enumerate}
\end{fact}

Therefore, to resolve (a), we can iterate over all pairs $(H,x) \in \mathscr{H} \times \Q$, testing condition (ii) in Fact~\ref{fact_P_ver}.  For each such pair satisfying (ii), let $X$ be the $\llex$-least game with $\Phi(X) = x$; then $H+X$ is a candidate for the $\llex$-least failure.  This yields at most $|\mathscr{H}| \cdot |\mathcal{Q}|$ candidates, over which we simply minimize.

Our technique for (b) uses the following central idea.

\begin{definition}
Let $X,Y \in \A$, and suppose that $X \llex Y$, $\Phi(X) = \Phi(Y)$, and $\Phi''X \subset \Phi''Y$.  Then we say that $Y$ is \emph{subsumed} by $X$.
\end{definition}

Now if $Y$ is subsumed by $X$, then $Y$ cannot be the $\llex$-least failure of $(\Q,\P,\Phi)$, because \emph{if} $Y$ satisfies (b), then so does $X$.  Moreover, $Y + Z$ cannot be the $\llex$-least failure, for \emph{any} $Z \in \A$: once again, if $Y + Z$ satisfies (b), then so does $X + Z$; and necessarily $X + Z \llex Y + Z$.  Therefore, we can completely disregard any element of $\A$ containing $Y$ as a subword.

We may therefore traverse the set $\A$ in lexicographic order, pruning whenever we reach an $X \in \A$ that we know to be subsumed.  Since $\Q$ is finite and there are at most $|\Q| \cdot 2^{|\Q|}$ possibilities for $(\Phi(X),\Phi''X)$, this traversal is guaranteed to terminate (even though $\A$ is infinite).  The full procedure is summarized as Algorithm~\ref{algorithm:nverification}.

\begin{algorithm}[tbp]
\hrule\vspace{0.15cm}
\begin{algorithmic}[1]
\State $T \gets \emptyset$
\State \Return \Call{LexLeastNFailure}{$n$, $0$}
\Statex
\Procedure {LexLeastNFailure}{$d$, $X$}
\If{$d = 0$} \Comment{Done recursing; check this value of $X$}
  \If{$X = 0$}
    \State \Return \textsc{null} \Comment{$0$ must be special-cased in mis\`ere play}
  \EndIf
  \State $x \gets \Phi(X)$; $\mathcal{E} \gets \Phi''X$
  \If{$x \in \P$}
    \State \Return \textsc{null} \Comment{We are seeking $\mathscr{N}$-failures only}
  \Else
    \If{$\P \cap \mathcal{E} = \emptyset$}
      \State \Return $X$ \Comment{We've found the lex-least failure}
    \Else
      \State $T \gets T \cup \{(x,\mathcal{E})\}$ \Comment{Add a transition record for $(x,\mathcal{E})$}
      \State \Return \textsc{null} \Comment{Keep searching}
    \EndIf
  \EndIf
\Else \Comment{Continue recursing from depth $d$}
  \Loop
    \State $Y \gets$ \Call{LexLeastNFailure}{$d-1$, $X$}
    \If{$Y \neq \textsc{null}$}
      \State \Return $Y$ \Comment{Found the answer; return it}
    \EndIf
    \State $X \gets X + H_d$ \Comment{Add a heap $H_d$ to $X$}
    \State $x \gets \Phi(X)$; $\mathcal{E} \gets \Phi''X$
    \If{there exists $\mathcal{D} \subset \mathcal{E}$ with $(x,\mathcal{D}) \in T$}
      \State \Return \textsc{null} \Comment{$X$ is subsumed; return immediately}
    \EndIf
  \EndLoop
\EndIf
\EndProcedure
\end{algorithmic}
\vspace{0.15cm}\hrule
\caption{$\mathscr{N}$-Verification Algorithm.}
\label{algorithm:nverification}
\end{algorithm}

\subsection*{Recalibration}

Given an r.b.m.\ $\QP$, a map $\Phi : \mathscr{H} \to \Q$, and a lexicographically least failure $X \in \A$ (as certified by Verification), we must construct a new triple $(\Q',\P',\Phi')$ whose least failure~$Y$ is guaranteed to satisfy $Y >_\mathrm{lex} X$.  The construction must further guarantee that if $\Q(\A)$ is finite, and if $(\Q_1,\P_1,\Phi_1)$, $(\Q_2,\P_2,\Phi_2)$, $\ldots$ is a sequence of successive recalibrations, then $\Q(\A) = (\Q_n,\P_n)$ for sufficiently large $n$.


What follows is an informal description of \emph{MisereSolver}'s recalibration procedure.  The essential idea is to ``free'' all generators involved in the failure $X$, mapping them to suitable monoids of the form $\<t~|~t^{n+k}=t^n\>$.  This gives an expanded monoid, in which the true outcome of $X$ may be ``marked'' without affecting the indicated outcome of any $Y \llex X$.  The expanded monoid is then immediately collapsed back down to a r.b.m.  This ensures that we can recalibrate as many times as necessary while keeping the candidate quotients reasonably small.

Given a candidate $(\Q,\P,\Phi)$ with $\llex$-least failure $X$:
\begin{enumerate}
\item Let $\mathscr{J} = \{J_1,\ldots,J_m\} \subset \mathscr{H}$ be the set of generators appearing in $X$.  Let $\Q^-$ be the submonoid of $\Q$ generated by $\{\Phi(H) : H \in \mathscr{H} \setminus \mathscr{J}\}$.  For each~$i$ ($1 \leq i \leq m$), let $\mathcal{R}_i = \<t~|~t^{n+k}=t^n\>$, for some carefully chosen~$k,n$.\footnote{\emph{MisereSolver}'s method of choosing such $k,n$ is currently rather crude and is not worth an extended discussion.  The important thing is that $k$ and $n$ must be large enough to arrive eventually at the correct presentation, but small enough to keep the computations tractable.  There is much room for improvement here.}
Now put
\[\Q^* = \Q^- \times \mathcal{R}_1 \times \mathcal{R}_2 \times \cdots \times \mathcal{R}_m.\]
\item Define $\Phi^* : \mathscr{H} \to \Q^*$ as follows.  For each $i$, put $\Phi^*(J_i) =$ the generator of $\mathcal{R}_i$.  For $H \not\in \mathscr{J}$, put $\Phi^*(H) = (\Phi(H),1,1,\ldots,1)$.
\item Define $\P^* \subset \Q^*$ as follows.  First put $1 \not\in \P^*$.  Then for each $x \in \Q^*$, let~$X$ be the $\llex$-least element of $\A$ with $\Phi^*(X) = x$.  Since $X' \llex X$ for each option $X'$ of $X$, we may assume that we have already defined whether each $\Phi^*(X') \in \P^*$.  So define
\[x \in \P^* \Longleftrightarrow \Phi^*(X') \not\in \P^* \textrm{ for each option } X'.\]
Essentially, we are guessing whether each such $X$ is a $\mathscr{P}$-position, under the assumption that $\Phi^*$ is correct on all its options.
\item This gives a bipartite monoid $(\Q^*,\P^*)$.  Compute its reduction $(\Q',\P')$ and let $\Phi'$ be the factor map.
\end{enumerate}

There is a great deal of inefficiency in this process: much more of the monoid is ``freed'' than is needed to enforce the correct outcome of~$X$.  Nonetheless, it works quite well for our purposes, and it satisfies the following two crucial theorems, which we state without proof.

\begin{theorem}[Recalibration Theorem]
Let $\QP$ be a r.b.m.\ and $\Phi : \A \to \Q$ a homomorphism with $\llex$-least failure $X$.  Let $(\Q',\P',\Phi')$ be the recalibration of $(\Q,\P,\Phi)$ by $X$ and let $Y$ be the lex-least failure of $\Phi'$.  Then $X \llex Y$.
\end{theorem}

\begin{theorem}[Termination Theorem]
Let $\A$ be a closed set of games with \emph{finite} mis\`ere quotient $\QP$ and pretending function $\Phi : \A \to \Q$.  Let $$(\Q_1,\P_1,\Phi_1),(\Q_2,\P_2,\Phi_2),\ldots$$ be a sequence of candidate quotients for $\A$, with each $(\Q_{n+1},\P_{n+1},\Phi_{n+1})$ computed by recalibrating on the $\llex$-least failure of $(\Q_n,\P_n,\Phi_n)$.  Then for all sufficiently large $n$, we have $(\Q_n,\P_n,\Phi_n) = (\Q,\P,\Phi)$.
\end{theorem}

In the raw form described here, Recalibration can be extremely slow.  If the least failure $X$ involves many generators, the expanded monoid $\Q^*$ is often quite large.  Fortunately, a tremendous shortcut is available: in step 1, rather than expand along the entire set $\mathscr{J}$, we can first try to expand along some proper subset $\mathscr{J}' \subsetneq \mathscr{J}$.  This almost always succeeds---in many cases, when $\mathscr{J}'$ is a singleton.

\subsection*{Optimizations for Heap Games}

\emph{MisereSolver}'s chief application is to calculate mis\`ere quotients of heap games.  When $\Gamma$ is an octal game, it is likely that $\Q_{n+1}(\Gamma) = \Q_n(\Gamma)$ for many values of~$n$.  In such cases \emph{MisereSolver} uses significant optimizations to compute the next pretension $\Phi(H_{n+1})$.

First, whenever a new partial quotient $\Q_n(\Gamma)$ is computed, \emph{MisereSolver} also computes the meximal sets $\mathcal{M}_x$ for each $x \in \Q_n$.  It also computes the antichain of lower bounds of the transition algebra $T_n$.  (See Section~5 for discussion.)  The algorithm used to compute this antichain is virtually identical to Algorithm~\ref{algorithm:nverification}.

Then, before computing $\Q_{n+1}(\Gamma)$, \emph{MisereSolver} first calculates the set $\mathcal{E} = \Phi''H_{n+1}$.  If there exists an $x \in \Q_n$ and a lower bound $(x,\mathcal{D}) \in T_n$ such that $\mathcal{D} \subset \mathcal{E} \subset \mathcal{M}_x$, then by the Mex Interpolation Principle it follows that $\Phi(H_{n+1}) = x$.  This check is extremely fast.  Furthermore, by the strong form of the Mex Interpolation Principle, it is not necessary to update the information about~$T_n$.  Thus \emph{MisereSolver} can run quickly through a large number of interpolated heaps.

If no such lower bound $(x,\mathcal{D})$ exists, it might still be the case that $\Q_{n+1}(\Gamma) = \Q_n(\Gamma)$.  To test this, \emph{MisereSolver} tries every $x$ with $\mathcal{E} \subset \mathcal{M}_x$.  For each such~$x$, the software assigns $\Phi(H_{n+1}) = x$ and runs Algorithm~\ref{algorithm:nverification}.  If the algorithm returns \textsc{null} for some value of $x$, then we are done.  (However, we must recompute the antichain of lower bounds of $T_{n+1}$.)  If it fails on every $x$, then \emph{MisereSolver} proceeds with the full recalibration/\allowbreak verification procedure described above.

\bibliography{games}

\begin{thebibliography}{1}

\bibitem{allemang_1984}
D.~T. Allemang.
\newblock Machine computation with finite games.
\newblock Master's thesis, Trinity College, Cambridge, 1984.
\newblock \\ \url{http://www.miseregames.org/allemang/}.

\bibitem{conway_1976}
J.~H. Conway.
\newblock {\em On Numbers and Games}.
\newblock A. K. Peters, Ltd., Natick, MA, second edition, 2001.

\bibitem{plambeck_2005}
T.~E. Plambeck.
\newblock Taming the wild in impartial combinatorial games.
\newblock {\em INTEGERS: The Electr. J. Combin. Number Thy.}, 5(\#G05), 2005.

\bibitem{miseregames_www}
T.~E. Plambeck and A.~N. Siegel.
\newblock {M}is\`ere {G}ames on the {W}eb.
\newblock \\ \mbox{\url{http://www.miseregames.org/}}.

\bibitem{siegel_200Xd}
T.~E. Plambeck and A.~N. Siegel.
\newblock Mis\`ere quotients for impartial games.
\newblock {F}orthcoming. \raggedright
  \mbox{\url{http://arxiv.org/abs/math.CO/0609825}}.

\end{thebibliography}

\end{document}